\documentclass{article}%

\usepackage{graphicx}
\usepackage{amsmath,amssymb,amsfonts}%
\usepackage{theorem}%
\usepackage{color}%
\usepackage{hyperref}
\usepackage[font={small, it}]{caption}
\theoremstyle{change}%
\newtheorem{definition}{Definition:}[section]%
\newtheorem{proposition}[definition]{Proposition:}%
\newtheorem{theorem}[definition]{Theorem:}%
\newtheorem{lemma}[definition]{Lemma:}%
{\theorembodyfont{\rmfamily}\newtheorem{remark}[definition]{Remark:}}%
{\theorembodyfont{\rmfamily}}%

\newenvironment{proof}
{{\bf Proof:}}
{\qquad \hspace*{\fill} $\Box$}%

\newcommand{\tr}{\operatorname{tr}}%

\newcommand{\inner}{\operatorname{int}}%
\newcommand{\rme}{\mathrm{e}}%

\newcommand{\CC}{\mathcal{C}}%
\newcommand{\OC}{\mathcal{O}}%
\newcommand{\UC}{\mathcal{U}}%
\newcommand{\N}{\mathbb{N}}%
\newcommand{\R}{\mathbb{R}}%
\newcommand{\Z}{\mathbb{Z}}%

\begin{document}

\title{Control sets of linear control systems on $\R^2$. The complex case.}
\author{V\'{\i}ctor Ayala \thanks{ Supported by Proyecto Fondecyt $n^{o}$ 1190142,
Conicyt, Chile} \\Universidad de Tarapac\'a\\Instituto de Alta Investigaci\'on\\Casilla 7D, Arica, Chile
\\and\\Adriano Da Silva and Erik Mamani \\Departamento de Matem\'atica,\\Universidad de Tarapac\'a - Arica, Chile.
}
\date{\today }
\maketitle

\begin{abstract}
  This paper explicitly computes the unique control set $D$ with non-empty interior of a linear control system on $\R^2$, when the associated matrix has complex eigenvalues. It turns out that the closure of $D$ coincides with the the region delimited by a computable periodic orbit $\OC$ of the system.
\end{abstract}

\section{Introduction}
Let A be a real matrix of order two. A linear control system (LCS) on $\R^2$ is given by the family of ODEs
	\begin{flalign*}
		&&\dot{v} = Av+u\eta, \;\;\;\;u\in\Omega,&&\hspace{-1cm}\left(\Sigma_{\R^2}\right)
	\end{flalign*}
where $\Omega=[u^-, u^+]$ with $u^-<u^+$ and $\eta\neq 0$. 

This article explicitly describes a maximal region $D$ of the system in which interior the controllability property holds. This region, called a control set, is relevant in applications. In fact, two arbitrary states in its interior can be connected by an integral curve of the system in positive time. In particular, by following an appropriate trajectory, it is possible to transform an initial condition into the desired state through the system in a finite time. Additionally, the existence of an optimal solution is also a warranty for a minimum time problem between these states.

Due to the exciting mathematical theory involved \cite{Agrachev}, \cite{Jurdjevic}, \cite{Kalman}, \cite{Pontryagin}, \cite{Wonham}; and the number of relevant applications \cite{Axelby}, \cite{Aseev}, \cite{Ledzewick}, \cite{Leitmann}, \cite{Reeds}, linear and non-linear control systems have been developed for more than 70 years. However, there is no literature for an arbitrary matrix $A$ for this particular system. 

Our approach is novel, and here we consider the drift $A$ with a couple of complex eigenvalues. We describe the corresponding control set by the different possibilities of $A$'s trace. And, we prove that $D$ is limited by a specific periodic orbit $\OC$ of the system. 

For a linear control system, it is well known that the Kalman rank condition warrants the existence of a control set $D$ with a non-empty interior. Furthermore, $D$ is characterized by the positive and negative orbits, which allows for determining some topological properties of $D$. However, computing these orbits is a difficult task, and the same is true for $D$. 

This article's main contribution through our approach, permits us to recover all the known results about the controllability and control sets properties for this class of systems without the extra assumptions $0\in\inner\Omega$. Moreover, the most crucial issue is to compute the control set $D$ explicitly as follow. The set,
$$\OC:=\left\{\varphi(s, P^+, u^-), s\in \left[0, \frac{\pi}{\mu}\right]\right\}\cup \left\{\varphi(s, P^-, u^+), s\in \left[0, \frac{\pi}{\mu}\right]\right\},$$
is a periodic orbit of $\Sigma_{\R^2}$.
Here, the points $P^-, P^+$ belong to $\R \cdot A^{-1}\eta$, a line determined by the drift and the control vector of the system. 
This orbit is obtained asymptotically by considering a solution starting on an equilibrium whose control function interchanges from $u^-$ and $u^+$. With that there are three possibilities: $\tr A=0$ and the system is controllable; $\tr A<0$ and $D$ is unique, closed and its boundary is $\OC$, or $\tr A>0$ and $D$ is open and coincides with the region bounded by $\OC$. Moreover, $\Sigma_{\R^2}$ admits only $D$ as control set when $\tr A<0$ and admits $D$ and $\OC$ as control sets when $\tr A>0$.

The article concludes with an asymptotic analysis through the parameters determining the dynamic of the system, i.e., the eigenvalues, and the size of the range determined by the controls $u^-$ and $u^+$. In particular, controllability properties are recovered in some cases. We also mention that our method does not consider $0$ in the range's interior as usual.

\bigskip

{\bf Notations:} For any vector $v\in\R^2$ we denote by $\R\cdot v$ the line passing by the origin and parallel to $v$. We consider the natural order on $\R\cdot v$ as 
$$v_1, v_2\in\R\cdot v, \;\;\; v_1\leq v_2\;\;\;\iff\;\;\; v_1-v_2=\alpha v, \;\;\alpha>0.$$
For any $\tau\in\R$, we denote by $R_{\tau}$ the rotation of $\tau$-degrees which is clockwise if $\tau<0$ and counter-clockwise if $\tau>0$. In particular, we use define $\theta:=R_{\pi/2}$.

\section{Geometric properties of spirals in $\R^2$.}

This section analyzes the dynamics of spirals in the Euclidean space $\R^2$. In particular, we show that spirals with the center in the same line have a particular kind of invariance.

Let $A\in \mathfrak{gl}(2, \R)$ and denote by $\sigma_A$ the number 
$$\sigma_A:=(\tr A)^2-4\det A.$$
The number $\sigma_A$ is related to the eigenvalues of $A$, and it is straightforward to see that $A$ has a pair of complex eigenvalues if and only if $\sigma_A<0$.  

\begin{definition}
For any $A\in\mathfrak{gl}(2, \R)$ with $\sigma_A<0$ we define the spiral $\varphi_A$ to be the function
$$(\tau, v_1, v_2)\in\R\times(\R^2\times\R^2)\setminus\Delta\mapsto \varphi_A(\tau, v_1, v_2):=\rme^{\tau A}(v_1-v_2)+v_2,$$
where $\Delta\subset \R^2\times\R^2$ is the diagonal. 
\end{definition}

Since $\sigma_A<0$, there exists an orthonormal basis of $\R^2$ such that 
$$A=\left(\begin{array}{cc}
  \lambda  &  -\mu \\
    \mu & \lambda
\end{array}\right), \;\;\;\;\mbox{ where }\;\;\;2\lambda=\tr A\;\;\;\mbox{ and }\;\;\;\mu^2=|\sigma_A|.$$
Consequently, the spiral $\varphi_A$ can be written on such basis, as
$$\varphi_A(\tau, v_1, v_2)=\rme^{\tau\lambda}R_{\tau\mu}(v_1-v_2)+v_2,$$
where $R_{\mu\tau}$ is the rotation of $\mu\tau$-degrees with relation to the previous basis, which is clockwise if $\mu\tau<0$ and counter-clockwise if $\mu\tau>0$.

The spiral $\varphi_A$ intersects the line passing by $v_1$ and $v_2$ for any $\tau\in k\frac{\pi}{\mu}\Z$. Moreover,
$$|\varphi_A(\tau, v_1, v_2)-v_2|=\rme^{s\lambda}|v_1-v_2|,$$
showing that $\varphi_A(\tau, v_1, v_2)$ belongs to the circumference with center $v_2$ and radius $\rme^{\tau\lambda}|v_1-v_2|$. In particular, 
$$\varphi_A(\tau, v_1, v_2)\rightarrow v_2\;\;\;\;\mbox{ when }\;\;\;\;\tau\lambda\rightarrow-\infty.$$

\begin{remark}
Note that, by reverting the time, we can relate the spirals associated with $\lambda$ and $-\lambda$. Also, if $B:\R^2\rightarrow\R^2$ is the linear map whose matrix on the previous basis is 
$$B=\left(\begin{array}{cc}
  1  &  0 \\
    0 & -1
\end{array}\right)\;\;\;\;\mbox{ then }\;\;\;\;B^2=I_{\R^2}\;\;\mbox{ and }\;\;BAB=\left(\begin{array}{cc}
  \lambda &  \mu \\
    -\mu & \lambda
\end{array}\right),$$
implying that the spirals associated with $\mu$ and $-\mu$ are related by conjugation.
\end{remark}

By the previous Remark let us assume w.l.o.g. that $\lambda<0$ and $\mu>0$ and consider $v_1, v_2\in(\R^2\times\R^2)\setminus\Delta$. Denote by $\CC_A(v_1, v_2)$, the region (see Figure \ref{fig1}) delimited by the line passing through $v_1$ and $v_2$, and the curve
$$
\left\{\varphi_A(\tau, v_1, v_2) \;\;\; \tau\in \left[0, \frac{\pi}{\mu}\right]\right\}.
$$

\begin{figure}[h!]
	\centering
	\includegraphics[scale=.6]{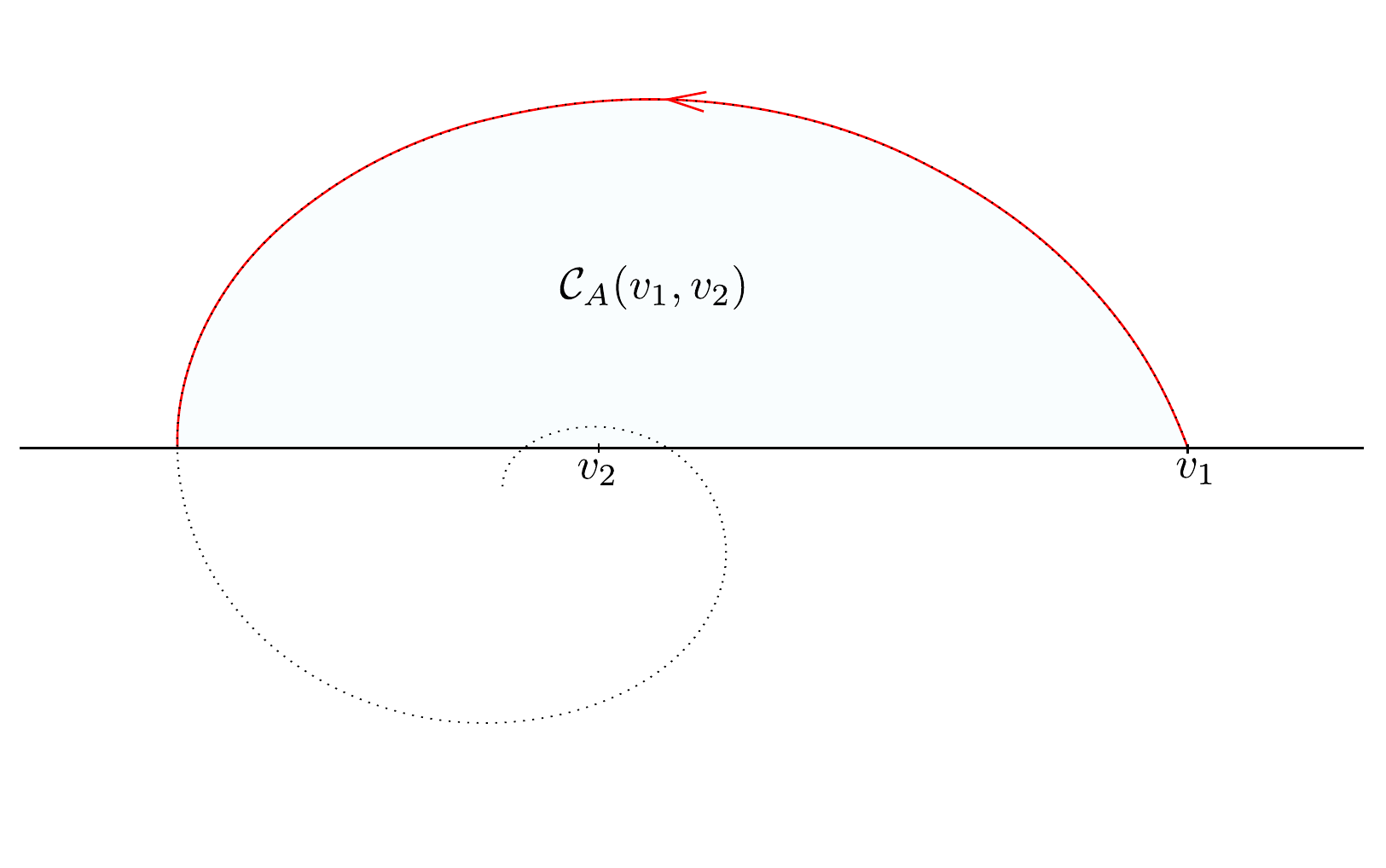}
	\caption{The region $\CC_A(v_1, v_2)$}
	\label{fig1}
\end{figure}

By our choices, such a region can be algebraically described as  
$$
\CC_A(v_1, v_2)=\left\{v\in\R^2; \;\;\langle v-v_2, \theta (v_1-v_2)\rangle\geq 0\;\;\mbox{ and }\;\;\langle v-\varphi_A(\tau, v_1, v_2), \theta A\rme^{\tau A}(v_1-v_2)\rangle\geq 0, \;\forall\tau\in\left[0, \frac{\pi}{\mu}\right]\right\},
$$
where $\theta$ is the counter-clockwise rotation of $\pi/2$-degrees.  The next result analyzes a kind of invariance for the region $\CC_A(v_1, v_2)$.

\begin{proposition}
\label{invariance}
For any $w_2\in [v_1, v_2]$ and $w_1\in \CC_A(v_1, v_2)$, it holds that 
$$\varphi_A(s, w_1, w_2)\in\CC_A(v_1, v_2), \;\;s\in \left[0, \frac{\pi-\sigma}{\mu}\right],$$
where $\sigma\in [0, \pi]$ is the angle between $v_1-v_2$ and $w_1-w_2$. Here, we are assuming that $\lambda<0$ and $\mu>0$.
\end{proposition}

\begin{proof}
Since,
$$\varphi_A(\tau, v_1, v_2)=\varphi_A(\tau, v_1-v_2, 0)+v_2,$$
the region $\CC_A(v_1, v_2)$ is obtained from $\CC_A(v_1-v_2, 0)$ through a translation by $v_2$. Therefore, it is enough to show the result assuming that $v_2=0$. For this case, we have that 
$$
\CC_A(v_1, 0)=\left\{v\in\R^2; \;\;\langle v, \theta v_1\rangle\geq 0\;\;\mbox{ and }\;\;\langle v-\rme^{\tau\lambda}R_{\tau\mu} v_1, \theta A\rme^{\tau A}v_1\rangle\geq 0, \;\forall\tau\in\left[0, \frac{\pi}{\mu}\right]\right\}.
$$
Moreover, in this case $w_2\in (0, v_1)$ and $\sigma$ is the angle between $v_1$ and $w_1-w_2$. Thus, we already have that,
$$\langle \varphi_A(s, w_1, w_2), \theta v_1\rangle=\langle \rme^{s A}(w_1-w_2)+w_2, \theta v_1\rangle\stackrel{w_2\in(0, v_2)}{=}\rme^{s\lambda}\langle R_{s\mu}(w_1-w_2), \theta v_1\rangle$$
$$=\rme^{s\lambda}\frac{|w_1-w_2|}{|v_1|}\langle R_{\mu s+\sigma} v_1, \theta v_1\rangle=\rme^{s\lambda}|w_1-w_2||v_1|\sin(\mu s+\sigma)\geq 0, \;\;\;\mbox{ since }\;\;\; \mu s\in [0, \pi-\sigma],$$
showing that 
$$\langle w_1, \theta v_1\rangle\geq 0\;\;\implies\;\;\langle \varphi_A(s, w_1, w_2), \theta v_1\rangle\geq 0\;\;\forall s\in \left[0, \frac{\pi-\sigma}{\mu}\right].$$

Define now the function
$$g:C_{\sigma}\rightarrow\R, \;\;\;\;\;\;\;g(s, \tau):=\langle \varphi_A(s, w_1, w_2)-\rme^{\tau A}v_1, \theta A\rme^{\tau A}v_1\rangle,$$
where $C_{\sigma}:=\left[0, \frac{\pi-\sigma}{\mu}\right]\times \left[0, \frac{\pi}{\mu}\right]$.
In order to conclude the result, it is enough to show that $g$ is nonnegative, that is,  
$$\forall (s, \tau)\in C_{\sigma}, \;\;\;\; g(s, \tau)\geq 0,$$
which we will do in the next steps.

\begin{itemize}
    \item[\bf Step 1.:] $g$ is nonnegative on critical points in $\inner C_{\sigma}$;
    
    By simple calculations, we get that 
    $$\frac{\partial g}{\partial s}(s, \tau)=\langle A\rme^{s A}(w_1-w_2), \theta A\rme^{\tau A}v_1\rangle=\det A\,\rme^{(s+\tau)\lambda}\frac{|w_1-w_2|}{|v_1|}\langle R_{\mu s+\sigma} v_1, \theta R_{\mu\tau} v_1\rangle$$
    $$=\det A\,\rme^{(s+\tau)\lambda}\frac{|w_1-w_2|}{|v_1|}\sin(\mu(s-\tau)+\sigma)=0\;\;\;\stackrel{(s, \tau)\in \inner D_{\sigma}}{\iff}\;\;\; \mu s+\sigma=\mu\tau.$$
    
    Also, 
    $$0=\frac{\partial g}{\partial \tau}(s, \tau)=\langle \varphi_A(s, w_1, w_2)-\rme^{\tau A}v_1, \theta A^2\rme^{\tau A}v_1\rangle,$$
    if and only if, there exists $\gamma\in\R$ such that,
    $$\varphi_A(s, w_1, w_2)-\rme^{\tau A}v_1=\gamma A^2\rme^{\tau A}v_1.$$
    The relation $\mu s+\sigma=\mu\tau$ gives us that
    $$\varphi_A(s, w_1, w_2)=\rme^{s A}(w_1-w_2)+w_2=\rme^{\frac{-\sigma}{\mu}\lambda}\rme^{\tau A}R_{-\sigma}(w_1-w_2)+w_2=\rme^{\frac{-\sigma}{\mu}\lambda}\frac{|w_1-w_2|}{|v_1|}\rme^{\tau A}v_1+w_2,$$
    and hence 
    $$\langle \varphi_A(s, w_1, w_2)-\rme^{\tau A}v_1, \theta\rme^{\tau A}v_1\rangle=\langle w_2, \theta\rme^{\tau A}v_1\rangle=-\rme^{\tau\lambda}|w_2||v_1|\sin\mu\tau.$$
    On the other hand, 
    $$\langle A^2\rme^{\tau A}v_1, \theta\rme^{\tau A}v_1\rangle=2\lambda\mu\rme^{2\tau\lambda}|v_1|^2,$$
    implying that 
    $$\gamma=-\frac{-\rme^{-\tau\lambda}}{2\lambda\mu}\frac{|w_2|}{|v_1|}\sin\mu\tau\geq 0, \;\;\;\mbox{ since }\;\lambda<0.$$

    In particular, if $g$ admits a critical point $(s, \tau)\in\inner D_{\sigma}$, by the previous arguments, we get 
    $$g(s, \tau)=\langle \varphi_A(s, w_1, w_2)-\rme^{\tau A}v_1, \theta A\rme^{\tau A}v_1\rangle=\gamma\langle A^2\rme^{\tau A}v_1, \theta A\rme^{\tau A}v_1\rangle$$
    $$=\gamma\det A\,\rme^{2\tau\lambda}\langle Av_1, \theta v_1\rangle=\gamma\det A\,\rme^{2\tau\lambda}\mu|v_1|^2\geq 0,$$
    showing the assertion.
    
    \item[\bf Step. 2:] $g$ is nonnegative on $\partial C_{\sigma}$.
    
    Let us start by noticing the point $\varphi_A\left(\frac{\pi-\sigma}{\mu}, w_1, w_2\right)$ belongs to the line $\R v_1$ and that
    $$\varphi_A\left(\frac{\pi-\sigma}{\mu}, w_1, w_2\right)=\rme^{\frac{\pi-\sigma}{\mu}A}(w_1-w_2)+w_2=-\rme^{\frac{\pi-\sigma}{\mu}\lambda}R_{-\sigma}(w_1-w_2)+w_2$$
    $$=-\rme^{\frac{\pi-\sigma}{\mu}\lambda}\frac{|w_1-w_2|}{|v_1|}v_1+w_2=\left(1-\rme^{\frac{\pi-\sigma}{\mu}\lambda}\frac{|w_1-w_2|}{|w_2|}\right)w_2,$$
    showing that $\varphi_A\left(\frac{\pi-\sigma}{\mu}, w_1, w_2\right)\leq w_2\leq v_1$. On the other hand, if $\beta\in [0, \pi]$ is the angle between $v_1$ and $w_1$ we obtain that $\beta<\sigma$ (see Figure \ref{fig2}) and 
    $$\rme^{\frac{\beta}{\mu}A}v_1=w_1\;\;\;\implies \;\;\;\rme^{\frac{\pi}{\mu}\lambda}|v_1|=\rme^{\frac{\pi-\beta}{\mu}\lambda}\left|\rme^{\frac{\beta}{\mu}A}v_1\right|= \rme^{\frac{\pi-\beta}{\mu}\lambda}|w_1|.$$
    Consequently, 
    $$|w_2|-\rme^{\frac{\pi-\sigma}{\mu}\lambda}|w_1-w_2|+\rme^{\frac{\pi}{\mu}\lambda}|v_1|= |w_2|-\rme^{\frac{\pi-\sigma}{\mu}\lambda}|w_1-w_2|+\rme^{\frac{\pi-\beta}{\mu}\lambda}|w_1|$$
    $$=\rme^{\frac{\pi-\beta}{\mu}\lambda}\left(|w_1|+\rme^{-\frac{\pi-\beta}{\mu}\lambda}|w_2|-\rme^{\frac{-(\sigma-\beta)}{\mu}\lambda}|w_1-w_2|\right)\geq \rme^{\frac{\pi-\beta}{\mu}\lambda}\left(|w_1|+|w_2|-|w_1-w_2|\right)\geq 0,$$
    implying that 
    $$\varphi_A\left(\frac{\pi-\sigma}{\mu}, w_1, w_2\right)-\rme^{\frac{\pi}{\mu}A}v_1=\left(|w_2|-\rme^{\frac{\pi-\sigma}{\mu}\lambda}|w_1-w_2|+\rme^{\frac{\pi}{\mu}\lambda}|v_1|\right)\frac{v_1}{|v_1|}\geq 0,$$
    and allowing us to conclude that 
    $$\rme^{\frac{\pi}{\mu}A}v_1\leq \varphi_A\left(\frac{\pi-\sigma}{\mu}, w_1, w_2\right)\leq v_1\;\;\;\implies\;\;\; \varphi_A\left(\frac{\pi-\sigma}{\mu}, w_1, w_2\right)\in\CC_A(v_1, 0).$$
    Therefore, 
    $$\forall \tau\in\left[0, \frac{\pi}{\mu}\right], \;\;\;g(0, \tau)\geq 0\;\;\;\mbox{ and }\;\;\;g\left(\frac{\pi-\sigma}{\mu}, \tau\right)\geq 0.$$
    
    By the previous calculations, 
    $$\frac{\partial g}{\partial s}(s, \tau)=\det A \,\rme^{(s+\tau)\lambda}\frac{|w_1-w_2|}{|v_1|}\sin(\mu(s-\tau)+\sigma),$$
    implying that, 
    $$\forall s\in \left(0, \frac{\pi-\sigma}{\mu}\right),\;\;\;\;\;\frac{\partial g}{\partial s}(s, 0)>0\;\;\;\mbox{ and }\;\;\;\frac{\partial g}{\partial s}\left(s, \frac{\pi}{\mu}\right)<0.$$
    As a consequence, it follows that 
    $$\forall s\in \left(0, \frac{\pi-\sigma}{\mu}\right),\;\;\;\;\;g(s, 0)\geq g(0, 0)\geq 0\;\;\;\mbox{ and }\;\;\; g\left(s, \frac{\pi}{\mu}\right)\geq g\left(\frac{\pi-\sigma}{\mu}, \frac{\pi}{\mu}\right)\geq 0.$$
\end{itemize}

Since $C_{\sigma}$ is a compact subset and $g$ is smooth, the Weierstrass Theorem assures the existence of a global minimum for $g$ on $C_{\sigma}$. Since the possible candidates for such minimum were calculated in Steps 1. and 2. we conclude that $g$ is nonnegative on $C_{\sigma}$, ending the proof.

\end{proof}

\begin{figure}[h!]
	\centering
	\includegraphics[scale=.6]{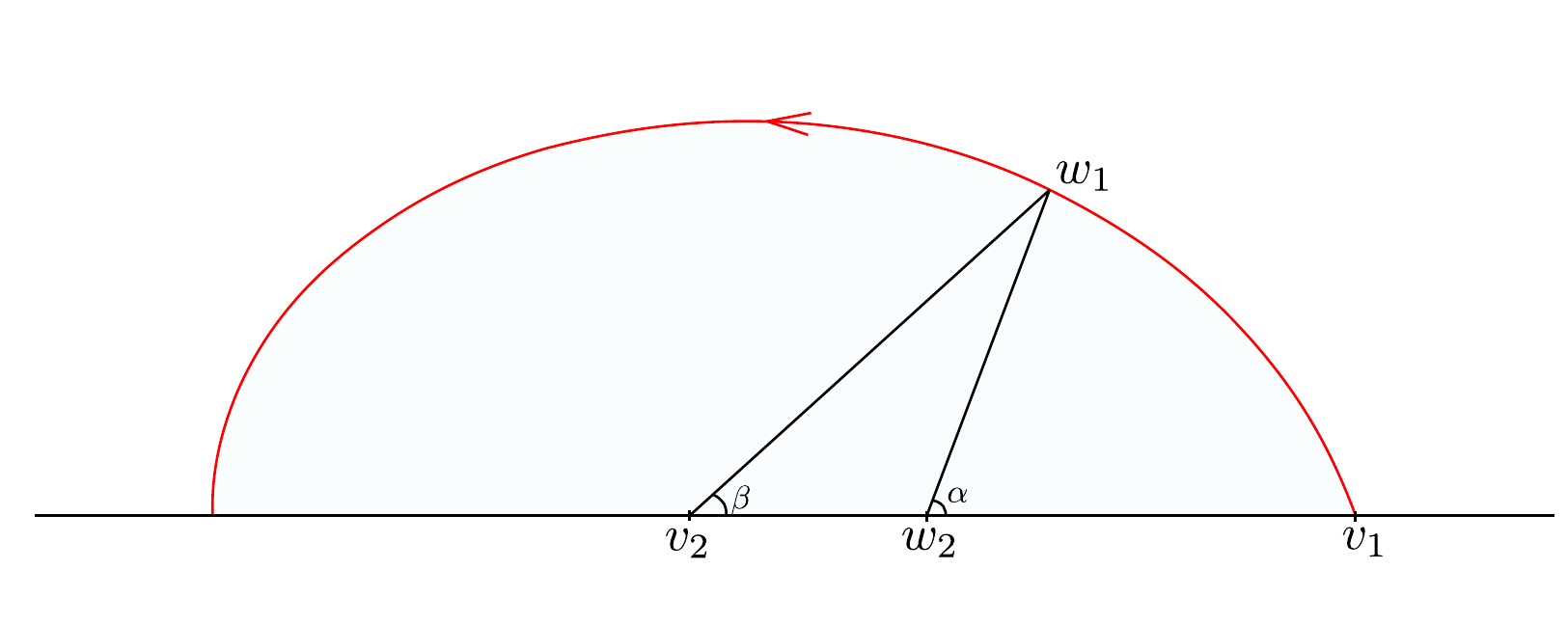}
	\caption{The invariance of $\CC_A(v_1, v_2)$}
	\label{fig2}
\end{figure}

\section{Linear control systems on $\R^2$}

A linear control system (LCS) on $\R^2$ is given by the family of ODEs
	\begin{flalign*}
		&&\dot{v}=Av+u\eta, \;\;\;\;u\in\Omega,&&\hspace{-1cm}\left(\Sigma_{\R^2}\right)
	\end{flalign*}
where $\Omega=[u^-, u^+]$ with $u^-<u^+$ and $\eta\neq 0$. 

The set $\Omega$ is called the {\it control range} of the system $\Sigma_{\R^2}$. The family of the {\it control functions} $\UC$ is, by definition, the set of all piecewise constant functions with image in $\Omega$. The {\it solution} of $\Sigma_{\R^2}$ starting at $v\in\R^2$ and associated control $ {\bf u}\in\UC$ is the unique piecewise differentiable curve $s\in\R\mapsto\varphi(s, v, {\bf u})$ satisfying 
$$\frac{d}{ds}\varphi(s, v, {\bf u})=A\varphi(s, v, {\bf u})+{\bf u}(s)\eta.$$
It is not hard to see that the solutions of $\Sigma_{\R^2}$ are given by concatenations of the curves associated with constant control functions.

For any $v\in\R^2$, the {\it positive} and the {\it negative orbits} of $\Sigma_{\R^2}$ are given, respectively, by the sets
$$\OC^+(v):=\{\varphi(s, v, {\bf u}), \;s\geq 0, {\bf u}\in\UC\}\hspace{1cm}\mbox{ and }\hspace{1cm}\OC^-(v):=\{\varphi(s, v, {\bf u}), \;s\leq 0, {\bf u}\in\UC\}.$$

\begin{definition}
\label{control}
A {\it control set} of $\Sigma_{\R^2}$ is a subset $D$ of $\R^2$ satisfying
\begin{itemize}
    \item[(a)] For any $v\in D$ there exists ${\bf u}\in\UC$ such that $\varphi(\R^+, v, {\bf u})\subset D$;
    \item[(b)] For any $v\in D$ it holds that $D\subset\overline{\OC^+(v)}$;
    \item[(c)] $D$ is maximal w.r.t. set inclusion satisfying (a) and (b). 
\end{itemize}

\end{definition}

If a control set $D$ of $\Sigma_{\R^2}$ satisfies $D=\R^2$ we say the $\Sigma_{\R^2}$ is controllable.

\begin{remark}
 Under the condition that $0\in\inner\Omega$, it is well know in the literatura that LCSs on Euclidean spaces admits a unique control set with nonempty interior. This control set is bounded if and only if the matrix $A$ is hyperbolic and is closed (open) if and only if $A$ has only eigenvalues with nonnegative (nonpositive) real parts (see for instance \cite[Chapter 3]{Fritz}).
\end{remark}
\pagebreak

From here we assume that the matrix $A\in\mathfrak{gl}(\R, 2)$ satisfies $\sigma_A<0$, and fix an orthonormal basis of $\R^2$ such that 
$$A=\left(\begin{array}{cc}
  \lambda  &  -\mu \\
    \mu & \lambda
\end{array}\right).$$

Since $\det A\neq 0$ it holds that
$$\varphi(s, v, u)=\rme^{s A}(v-v(u))+v(u), \;\;\;\mbox{ where } \;\;\;v(u)=-uA^{-1}\eta,$$
are the equilibria of the system. In particular, the solutions of $\Sigma_{\R^2}$ for constant control functions coincide with the spirals $\varphi_A(s, v, v(u))$ if $\tr A\neq 0$ and lie on circumferences if $\tr A=0$.

%If $\lambda\neq 0$ define 
 % $$\Omega_A:=\left[\frac{u^--u^+\rme^{\varepsilon\pi\frac{\lambda}{\mu}}}{1-\rme^{\varepsilon\pi\frac{\lambda}{\mu}}}, \frac{u^+-u^-\rme^{\varepsilon\pi\frac{\lambda}{\mu}}}{1-\rme^{\varepsilon\pi\frac{\lambda}{\mu}}}\right], \;\;\;\mbox{ where }\;\;\; \varepsilon:=\left\{\begin{array}{ll}
 % -1  &  \mbox{ if }\lambda>0\\
 %  1  & \mbox{ if }\lambda<0
%\end{array}\right..$$
%It is straightforward to see that $\Omega_A$ is a compact subset of $\R$ that contains $\Omega$ in its interior. The relation between $\Omega_A$ and the control sets of $\Sigma_{\R^2}$ is given in our main result.

In what follows we analyze the dynamics of the solutions of $\Sigma_{\R^2}$ in order to obtain a full characterization of the control sets of the system. Moreover, all the results that follows do not need the assumption that $0\in\inner\Omega$.

%\begin{theorem}
%Any LCS on $\R^2$ whose drift satisfies $\sigma_A<0$ admits a unique control set $\CC_{\R^2}$. 

%\end{theorem}

%The proof of the previous theorem will follow from the next sections. 

\subsection{The control set with nonempty interior}

In this section, we construct explicitly the control set of $\Sigma_{\R^2}$ with a non-empty interior by considering the possibilities for the trace of the matrix $A$.

\subsubsection{The case $\tr A=0$}

In this case, the solutions of $\Sigma_{\R^2}$ for constant controls have the form 
$$\varphi(s, v, u)=R_{s\mu}(v-v(u))+v(u),$$
and they lie on the circumferences $C_{u, v}$ with center $v(u)$ and radius $|v-v(u)|$.

\begin{theorem}
If the associated matrix $A$ of $\Sigma_{\R^2}$ is such that $\tr A=0$ and $\det A>0$, then $\Sigma_{\R^2}$ is controllable.
\end{theorem}

\begin{proof}
In order to show the result, it is enough to construct a periodic orbit between an arbitrary point $v\in\R^2$ and some fixed $v(u_0)\in v(\Omega)$, which we do as follows:

	\begin{itemize}
	    \item[(a)]  $v(\Omega)=[v(u^-), v(u^+)]$ is a compact interval on the line $\R\cdot\theta\eta$;
	
		\item[(b)] The circumference $C_{u^+, v}$ intersects the line $\R\cdot\theta\eta$ in two points. Denote by $v_1$ the point in this intersection close to $v(u^-)$. In particular, $v_1=\varphi(s_1, v, u^+)$ for some $s_1>0$;
		
		\item[(c)] If $v_1\notin v(\Omega)$, we repeat the process in the previous item for the circumference $C_{u^-, v_1}$, obtaining a point $v_2$. 
		
		\item[(d)] Repeating the previous process, if $v_n\notin v(\Omega)$, we obtain in the same way, a point $v_{n+1}$ belonging to the intersection of the circumference $C_{\bar{u}, v_n}$, and the line $\R\cdot\theta\eta$, where $\bar{u}=u^+$ if $n$ is even and $\bar{u}=u^-$ if $n$ is odd. By induction, we quickly see that the radius $R_n$ of $C_{\bar{u}, v_n}$ satisfies 
		$$R_n=|v_n-v(\bar{u})|=|v-v(u^-)|-n|v(u^+)-v(u^-)|.$$
		Therefore, there exists $N\in\N$ such that $v_N\in v(\Omega)$.
		
		\item[(e)] Now, since $v_N\in v(\Omega)$ there exists, by continuity, $u_N\in\Omega$ satisfying $|v(u_N)-v(u_0)|=|v_N-v(u_N)|.$ The circumference $C_{u_N ,v_N}$ passes through $v_N$ and by the point $v(u_0)$. Therefore, there exists $s_N>0$ such that $\varphi(s_N, v_N, u_N)=v(u_0)$ and by concatenation we get a trajectory from $v$ to $v(u_0)$ (blue paht in Figure \ref{fig5}).
		
		\item[(f)] By choosing the complementary path (red path in Figure \ref{fig5}) on the circumferences constructed on the previous items, we obtain a trajectory from $v(u_0)$ to $v$, which gives us a periodic orbit as desired (Figure \ref{fig4}).
		
	\end{itemize}

\end{proof}

\begin{figure}
	\centering
	\includegraphics[scale=1]{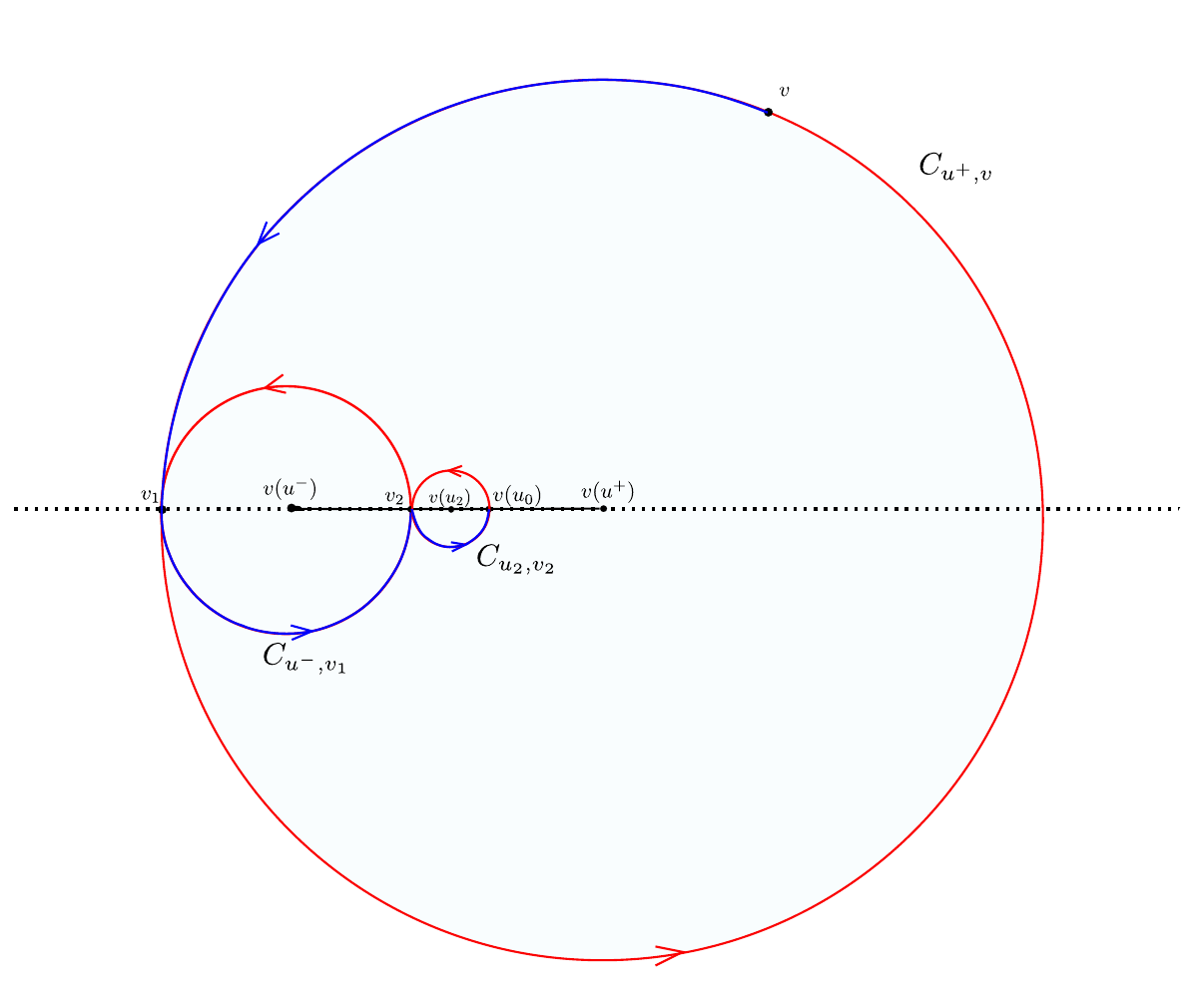}
	\caption{Periodic Orbit through $v(u_0)$ and $v$.}
	\label{fig5}
\end{figure}

\subsubsection{The case $\tr A\neq 0$}

Next, we construct a periodic orbit for $\Sigma_{\R^2}$. The main result in this section will show that such orbit is the boundary of the unique control set of $\Sigma_{\R^2}$. 

As previously w.l.o.g. that the eigenvalues of $A$ are $\lambda\pm\mu i$ with $\lambda<0$ and $\mu>0$. Define recurrently 
$$P_0=v(u^+),\;\;\;P_{2n+1}:=\varphi\left(\frac{\pi}{\mu}, P_{2n}, u^-\right)\;\;\;\mbox{ and }\;\;\;  P_{2n+2}:=\varphi\left(\frac{\pi}{\mu}, P_{2n+1}, u^+\right), \;\;\;n\geq 0.$$ 

A simple inductive process allows us to obtain 

$$P_{2n}=-\rme^{\pi\frac{\lambda}{\mu}}\left[\sum_{j=0}^{2n-1}\rme^{j\pi\frac{\lambda}{\mu}}\right]v(u^-) + \left[\sum_{j=0}^{2n}\rme^{j\pi\frac{\lambda}{\mu}}\right]v(u^+),\hspace{1cm}  n\geq 1$$
and 
$$P_{2n+1}=\left[\sum_{j=0}^{2n-1}\rme^{j\pi\frac{\lambda}{\mu}}\right]v(u^-) -\rme^{\pi\frac{\lambda}{\mu}}\left[\sum_{j=0}^{2n}\rme^{j\pi\frac{\lambda}{\mu}}\right]v(u^+),\hspace{1cm} n\geq 0.$$

On the other hand,  
$$\frac{\lambda}{\mu}<0\;\;\;\implies\;\;\;\rme^{2\pi\frac{\lambda}{\mu}}<1\;\;\;\implies\;\;\;\sum_{j=0}^{m}\rme^{j\pi\frac{\lambda}{\mu}}=\sum_{j=0}^{m}\left(\rme^{\pi\frac{\lambda}{\mu}}\right)^j\rightarrow \frac{1}{1-\rme^{\pi\frac{\lambda}{\mu}}}\;\;\;\;\mbox{ as }\;\;\;m\rightarrow+\infty.$$  

Consequently, 
$$P_{2n}\rightarrow P^+:=\left(\frac{-u^++\rme^{\pi\frac{\lambda}{\mu}}u^-}{1-\rme^{\pi\frac{\lambda}{\mu}}}\right)A^{-1}\eta\;\;\;\;\mbox{ and }\;\;\;\;P_{2n+1}\rightarrow P^-:=\left(\frac{-u^-+\rme^{\pi\frac{\lambda}{\mu}}u^+}{1-\rme^{\pi\frac{\lambda}{\mu}}}\right)A^{-1}\eta.$$

Note that both of the points $P^-, P^+$ belong to the line $\R A^{-1}\eta$ and satisfy
$$P^+-v(u^+)=\frac{(u^+-u^-)\rme^{\pi\frac{\lambda}{\mu}}}{u^+(1-\rme^{\pi\frac{\lambda}{\mu}})}v(u^+)\hspace{.5cm}\mbox{ and }\hspace{.5cm}P^--v(u^-)=-\frac{(u^+-u^-)\rme^{\pi\frac{\lambda}{\mu}}}{u^-(1-\rme^{\pi\frac{\lambda}{\mu}})}v(u^-),$$
implying that 
$$P^-<v(u^-)< v(u^+)< P^+\;\;\;\;\mbox{ on the line }\;\;\;\R\cdot(-A^{-1}\eta).$$

Moreover, it holds that 
$$\varphi\left(\frac{\pi}{\mu}, P^+, u_-\right)=-\rme^{\pi\frac{\lambda}{\mu}}P^++(1+\rme^{\mu\frac{\lambda}{\mu}})v(u^-)=\left[-\rme^{\pi\frac{\lambda}{\mu}}\left(\frac{-u^++\rme^{\pi\frac{\lambda}{\mu}}u^-}{1-\rme^{\pi\frac{\lambda}{\mu}}}\right)-(1+\rme^{\mu\frac{\lambda}{\mu}})u^-\right] A^{-1}\eta$$
$$\left(\frac{\rme^{\pi\frac{\lambda}{\mu}}u^+-\rme^{2\pi\frac{\lambda}{\mu}}u^--(1-\rme^{2\mu\frac{\lambda}{\mu}})u^-}{1-\rme^{\pi\frac{\lambda}{\mu}}} \right)A^{-1}\eta=\left(\frac{-u^-+\rme^{\pi\frac{\lambda}{\mu}}u^+}{1-\rme^{\pi\frac{\lambda}{\mu}}}\right)A^{-1}\eta=P^-,$$
and analogously, 
$$\varphi\left(\frac{\pi}{\mu}, P^-, u^+\right)=P^+,$$
showing the following:

\begin{proposition}
\label{periodic}
The subset of $\R^2$ given by 
$$\OC:=\left\{\varphi(s, P^+, u^-), s\in \left[0, \frac{\pi}{\mu}\right]\right\}\cup \left\{\varphi(s, P^-, u^+), s\in \left[0, \frac{\pi}{\mu}\right]\right\},$$
is a periodic orbit of $\Sigma_{\R^2}$.
\end{proposition}

\begin{figure}
	\centering
	\includegraphics[scale=.6]{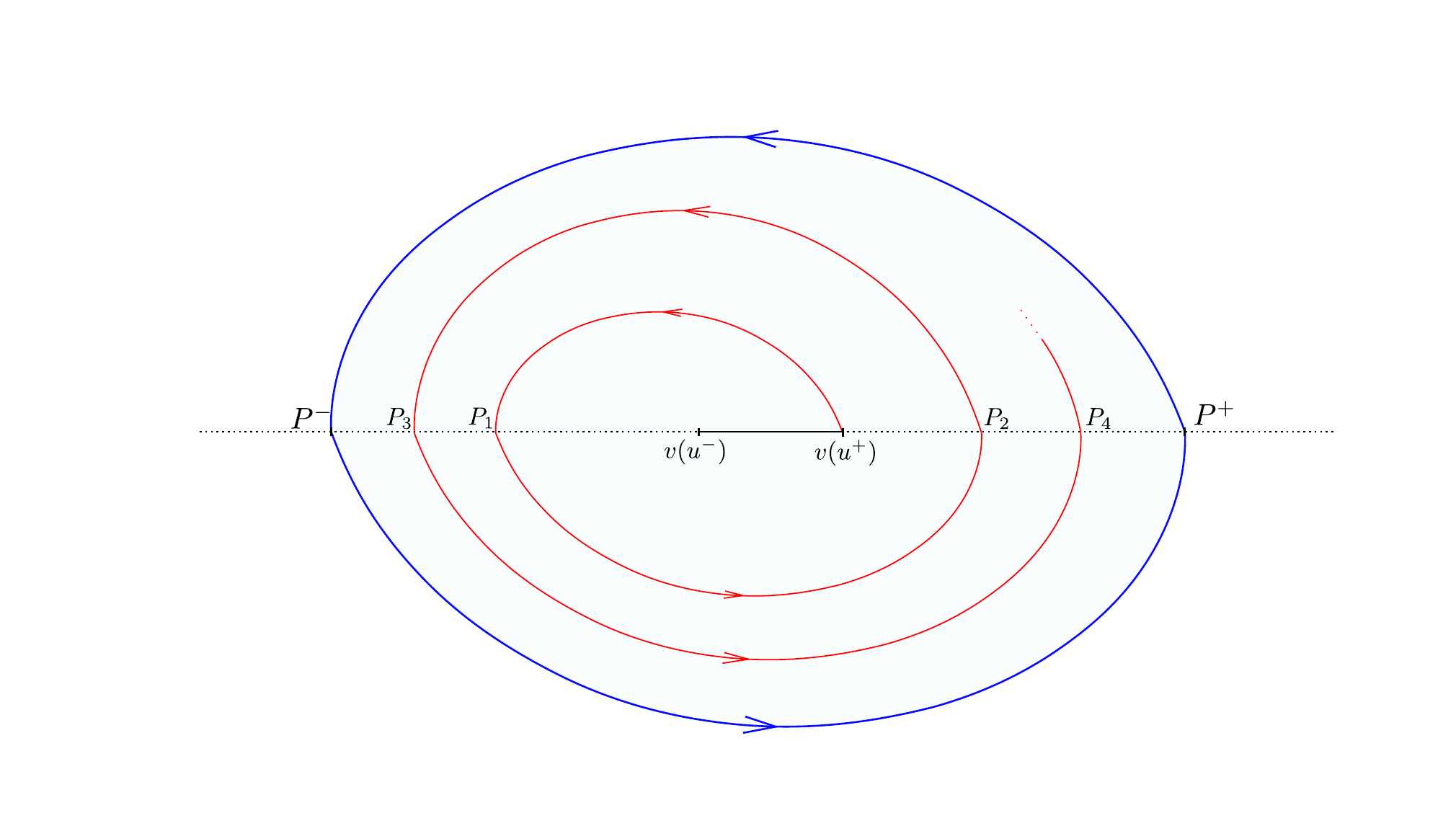}
	\caption{Periodic Orbit}
	\label{fig3}
\end{figure}

Let us denote by $\CC$ the closure of the region delimited by the periodic orbit $\OC$. The next result shows that $\CC$, or its interior, is a control set of the system.

\begin{theorem}
\label{teo}
For the LCS $\Sigma_{\R^2}$ with $\sigma_A<0$ and $\tr A\neq 0$ it holds that 
\begin{enumerate}
    \item $\tr A<0$ and $D=\CC$ is a control set;
    \item $\tr A>0$ and $D=\inner\CC$ is a control set.
\end{enumerate}
\end{theorem}

\begin{proof}  Let us start by showing, in the next steps, that  
$$\forall v\in\inner \CC, \;\;\;\inner \CC=\OC^+(v)\;\;\mbox{ if }\;\;\lambda<0\;\;\;\mbox{ and }\;\;\;\inner \CC=\OC^-(v)\;\;\mbox{ if }\;\;\lambda>0.$$ 
Since both cases are analogous, we will assume w.l.o.g. that $\lambda<0$ and $\mu>0$.

\begin{itemize}
    \item[\bf Step 1:] $\CC$ is positively invariant; 

For any $u\in\Omega$, it turns out 
$$\varphi(s, v, u)=\varphi_A(s, v, v(u)).$$
As a consequence, the region $\CC$ can be decomposed in two regions $$\CC_A(P^+, v(u^-))\;\;\;\mbox{ and }\;\;\;\CC_A(P^-, v(u^+)),$$
which are delimited, by the line passing through $v(u^+)$ and $v(u^-)$ and the curves 
$$\left\{\varphi(s, P^+, u^-), s\in \left[0, \frac{\pi}{\mu}\right]\right\}\;\;\;\;\mbox{ and }\;\;\;\left\{\varphi(s, P^-, u^+), s\in \left[0, \frac{\pi}{\mu}\right]\right\},$$ 
respectively. 

Moreover, on the line $\R\cdot(-A^{-1}\eta)$, for any $u\in\Omega$ and $w\in \CC_A(P^+, v(u^-))$, it holds that  $v(u)\in[v(u^-), P^+]$. Therefore, by Proposition \ref{invariance} it holds that 
$$\varphi(s, w, u)=\varphi_A(s, w, v(u))\in \CC_A(P^+, v(u^-)),$$
for any $s\in \left[0, \frac{\pi-\sigma}{\mu}\right]$. Here, $\sigma$ is the angle between $v(u^+)-v(u^-)$ and $w-v(u)$. In particular, 
$$P_1:=\varphi\left(\frac{\pi-\sigma}{\mu}, w, u\right)\in [P^-, P^+]\subset\CC_A(P^-, v(u^+)).$$
Since $P_1\in \CC_A(P^-, v(u^+))$ and $v(u)\in [v(u^+), P^-]$, Proposition \ref{invariance}
implies that 
$$\varphi(s, P_1, u)=\varphi_A(s, P_1, v(u))\in \CC_A(P^-, v(u^+)), $$
for any $s\in \left[0, \frac{\pi}{\mu}\right]$. Again, 
$$P_2:=\varphi\left(\frac{\pi}{\mu}, P_1, u\right)\in [P^-, P^+]\subset\CC_A(P^+, v(u^-)).$$
Since we can repeat the process, we already prove the invariance of $\CC$ in positive time. 

\item[\bf Step 2:] Controllability holds on $\inner \CC$;

 The result certainly follows if we show the relationships
 $$\forall v\in\inner \CC, \;\;\;v\in\OC^+(v(u^-))\;\;\;\mbox{ and }\;\;\;v(u^-)\in\OC^+(v).$$
The assumption $\lambda<0$ implies that 
$$\forall u\in\Omega, \;\;\;|\varphi(s, v, u)|\rightarrow+\infty\;\;\;\mbox{ as }\;\;\;s\rightarrow-\infty.$$
Consequently, the compactness of $\CC$ shows the existence of $s_0>0$ such that $\varphi(-s_0, v, u)\in\partial \CC=\OC$. Moreover, there exists $n\in\N$ and $t_0>0$ such that
$$\varphi(-s_0, v, u)=\varphi(t_0, P_{2n}, u^-) \;\;\;\;\; \mbox{ or  } \;\;\;\;\; \varphi(-s_0, v, u)=\varphi(t_0, P_{2n+1}, u^+).$$
By construction the points $P_{m}, m\in\N$ are attained from $v(u^-)$ in positive time. Therefore, the previous arguments show that $v$ is attained from $v(u^-)$, or equivalently $v\in\OC^+(v(u^-))$.
Furthermore, $s\mapsto \varphi(s, v, u^-)$ is a curve that revolves around $v(u^-)$ and $s\mapsto \varphi(s, v(u^-), u)$ revolves around $v(u)$. Then, for any $u\neq u^-$, there exist $s_0, t_0>0$ such that
$$\varphi(s_0, v, u^-)=\varphi(-t_0, v(u^-), u)\;\;\;\implies\;\;\; v(u^-)\in\OC^+(v),$$
proving the claim.

    \item[\bf Step 3:] It holds that 
    $$\forall v\in\inner \CC, \;\;\; \inner \CC=\OC^+(v).$$
    
    Since, for any $s\in\R$ and $u\in\Omega$, the map 
    $$v\in\R^2\mapsto\varphi(s, v, u)\in\R^2,$$
    is a diffeomorphism, Step 1 implies that 
    $$\varphi(s, \inner \CC, u)\subset\inner \CC, \;\;\;\forall s>0, u\in\Omega.$$
    As a consequence, 
    $$\forall v\in\inner \CC, \;\;\; \OC^+(v)\subset\inner \CC.$$
    On the other hand, by Step 2, controllability holds inside $\inner \CC$. Consequently, for any $v, w\in\inner \CC$ we obtain 
    $$w\in\OC^+(v)\;\;\;\implies\;\;\;\inner \CC\subset\OC^+(v),$$
showing the desired.
\end{itemize}

By the previous, it is straightforward to see that $\inner \CC$ satisfies conditions (a) and (b) of Definition \ref{control}. Therefore, there exists a control set $D$ such that $\inner \CC\subset D$ and we have that:
\begin{itemize}
    \item[1.] If $\lambda<0$, the positively invariance on item (a) implies that $\OC^+(v)\subset \CC$ for all $v\in \CC$. Since $v\in D$, condition (b) in Definition \ref{control} implies that $D\subset\overline{\OC^+(v)}$ and hence
    $$\overline{\OC^+(v)}\subset \overline{\CC}=\CC\subset D\subset \overline{\OC^+(v)},$$
    showing that $D=\CC$ is in fact the control set of $\Sigma_{\R^2}$.
    
    \item[2.] If $\lambda>0$ let $v\in\R^2$ and assume that $$\overline{\OC^+(v)}\cap\inner \CC\neq \emptyset.$$
    In particular, there exists $s>0$, ${\bf u}\in\UC$ such that $$\varphi(s, v, {\bf u})\in\inner \CC\;\;\implies\;\;v\in\varphi(-s, \inner \CC, {\bf u}')\subset\inner \CC,$$
    implying the maximality of $\inner \CC$ and hence $D=\inner \CC$, concluding the proof.
\end{itemize}
\end{proof}

\begin{remark}
The previous result implies that, if $\tr A\neq 0$, the LCS admits a bounded control set with nonempty interior which is closed if $\tr A<0$ and open when $\tr A>0$. Moreover, from Step 1 in the proof of Theorem \ref{teo}, it holds that
\begin{equation}
\label{eq}
    \forall v\in\CC, {\bf u}\in\UC \;\;\;\;\varphi(s, v, {\bf u})\in\CC\;\;\;\mbox{ if }\;\;\; s\cdot\tr A<0
\end{equation}
 
\end{remark}

\subsection{The possible control sets of a LCS}

As is well stated in the literature, if $0\in\inner\Omega$, the control set $D$ previously obtained is the only control set of $\Sigma_{\R^2}$ with non-empty interior. This section shows that $D$ is in fact the only control set with non-empty interior, even without the condition $0\in\inner\Omega$. Moreover, if the trace of the associated matrix $A$ is positive, the periodic orbit $\OC=\partial D$ is also a control set of $\Sigma_{\R^2}$.

In order to show the previous claim, the following statement will be crucial.

\begin{lemma}
For any $v\in\R^2\setminus \CC$, $u\in\Omega$ it holds that:
\begin{itemize}
    \item[(a)] $|\varphi(s, v, u)-\CC|\leq\rme^{s\lambda}|v-\CC|$\;\;\; if\;\;\; $s\lambda<0$;
    \item[(b)] $|\varphi(s, v, u)-\CC|\geq\rme^{s\lambda}|v-\CC|$ \;\;\;if\;\;\; $s\lambda>0$,
    where $2\lambda=\tr A$.
\end{itemize}

\end{lemma}

\begin{proof}
(a) Since $\CC$ is compact, for any $v\in\R^2$ there exists $v_0\in\CC$ such that $|v-\CC|=|v-v_0|$. By equation (\ref{eq}), it holds that
$$\lambda s<0\;\;\implies\;\;\varphi(s, v_0, u)\in\CC.$$
Consequently,
$$|\varphi(s, v, u)-\CC|\leq |\varphi(s, v, u)-\varphi(s, v_0, u)|=\rme^{s\lambda}|v-v_0|=\rme^{s\lambda}|v-\CC|,$$
showing the assertion.

(b) Let us assume the existence of $v_0\in\R^2\setminus \CC$, $u_0\in\Omega$ and $s_0\in \R$ such that 
$$\lambda s_0>0\;\;\;\mbox{ and }\;\;\;|\varphi(s_0, v_0, u_0)-\CC|< \rme^{s_0\lambda}|v_0-\CC|.$$
Since, 
$$w_0=\varphi(s_0, v_0, u_0)\;\;\;\iff\;\;\;v_0=\varphi(-s_0, w_0, u_0),$$
we have that 
$$w_0\in\CC \;\;\;\mbox{ and }\;\;\;-\lambda s_0<0\;\;\;\stackrel{(\ref{eq}}{\implies}\;\;\; v_0=\varphi(-s_0, w_0, u_0)\in \CC,$$
which cannot happens. Therefore, $w_0\in\R^2\setminus \CC$ and by item (a) we obtain
$$|\varphi(s_0, v_0, u_0)-\CC|< \rme^{s_0\lambda}|v_0-\CC|=\rme^{s_0\lambda}|\varphi(-s_0, w_0, u_0)-\CC|\stackrel{(a)}{\leq}\rme^{s_0\lambda}\rme^{-s_0\lambda}|w_0-\CC|=|\varphi(s_0, v_0, u_0)-\CC|,$$
which is absurd. Therefore, item (b) holds.

\end{proof}

We can now prove the main result concerning the control sets of a LCS on $\R^2$.

\begin{theorem}
Let $\Sigma_{\R^2}$ be a LCS satisfying $\sigma_A<0$ and $\tr A\neq 0$. It holds:
\begin{itemize}
    \item[1.] If $\tr A<0$ the only control set of $\Sigma_{\R^2}$ is $D$;
    \item[2.] If $\tr A>0$ then $\inner D$ and $\partial D$ are the only control sets of $\Sigma_{\R^2}$.
\end{itemize}
\end{theorem}

\begin{proof} Since $\partial D=\OC$ is a periodic orbit, it satisfies conditions (a) and (b) of Definition \ref{control} and is therefore contained in a control set of $\Sigma_{\R^2}$. By Theorem \ref{teo}, we know that $D=\CC$ is a control set if $\tr A<0$ and $D=\inner \CC$ is a control set if $\tr A>0$. 
Therefore, the result follows if we show that no control set of  $\Sigma_{\R^2}$ intersects $\R^2\setminus\CC$.

Since the solutions of $\Sigma_{\R^2}$ are given by concatenations of the solutions for constant controls, it is not hard to show by induction that for all ${\bf u}\in\UC$ and $v\in\R^2\setminus\CC$,
$$|\varphi(s, v, {\bf u})-\CC|\leq \rme^{s\lambda}|v-\CC|\;\;\;\mbox{ if }\;\;\;\lambda s<0,$$
and
$$|\varphi(s, v, {\bf u})-\CC|\geq \rme^{s\lambda}|v-\CC|\;\;\;\mbox{ if }\;\;\;\lambda s>0.$$

In particular, if $|v-\CC|=\epsilon>0$ we have that 
$$\OC^+(v)\subset N_{\epsilon}\left(\CC\right)\;\;\;\mbox{ if }\;\;\;\lambda<0\;\;\;\mbox{ and }\;\;\;\OC^+(v)\subset \R^2\setminus N_{\epsilon}\left(\CC\right)\;\;\;\mbox{ if }\;\;\;\lambda>0$$
Let us assume that $\Sigma_{\R^2}$ admits a second control set $D'$ satisfying $|v-\CC|=\epsilon>0$ for some $v\in D'$.

By condition (a) in Definition \ref{control}, there exists ${\bf u}\in\UC$ such that $\varphi(s, v, {\bf u})\in D'$ for all $s>0$. If $v\notin \CC$, we have by invariance (see equation (\ref{eq})), that $\varphi(s, v, {\bf u})\notin \CC$ for all $s>0$. Moreover, by condition (b) in Definition \ref{control} and the previous calculations, it holds that $D'\subset\overline{\OC^+(\varphi(s, v, {\bf u}))}$ and by the previous
$$D'\subset N_{\rme^{s\lambda}\epsilon}(\CC)\;\;\;\mbox{ if }\;\;\;\lambda<0\;\;\;\mbox{ and }\;\;\;D'\subset \R^2\setminus N_{\epsilon}\left(\CC\right)\;\;\;\mbox{ if }\;\;\;\lambda>0.$$
Consequently, for all $s>0$
$$|v-\CC|\leq \rme^{s\lambda}|v-\CC|\;\;\mbox{ if }\;\;\lambda<0\;\;\mbox{ and }\;\;\;|v-\CC|\geq \rme^{s\lambda}|v-\CC|\;\;\mbox{ if }\;\;\lambda>0,$$
which is not possible if $|v-\CC|\neq 0$. 

Therefore, any control set $D'$ of $\Sigma_{\R^2}$ satisfies $D'\subset \CC$ concluding the proof.

\end{proof}

\section{Remarks on continuity and asymptotic behavior of control sets}

The construction of the periodic orbit $\OC$ allows us to analyze the asymptotic behavior of the control set $D$ as the control range grows.

Fix a real number $\nu\in\R$ and define $\Omega_{\alpha, \rho}=[\alpha, \rho]$ where $\alpha<\nu<\rho$. Define the LCS
\begin{flalign*}
		&&\dot{v}=Av+u\eta, \;\;\;\;u\in\Omega_{\alpha, \rho},&&\hspace{-1cm}\left(\Sigma^{\alpha, \rho}_{\R^2}\right)
	\end{flalign*}
where $\eta\neq 0$ and the matrix $A$ satisfies $\sigma_A<0$ and $\tr A<0$. By Theorem \ref{teo} the LCS $\Sigma_{\R^2}^{\alpha, \rho}$ admits a unique control set with nonempty interior $D^{\alpha, \rho}$ whose boundary is the periodic orbit

$$\OC^{\alpha, \rho}:=\left\{\varphi(s, P^+_{\alpha, \rho}, \alpha), s\in \left[0, \frac{\pi}{\mu}\right]\right\}\cup \left\{\varphi(s, P^-_{\alpha, \rho}, \rho), s\in \left[0, \frac{\pi}{\mu}\right]\right\},$$
with 
$$P^+_{\alpha, \rho}:=\left(\frac{-\rho+\rme^{\pi\frac{\lambda}{\mu}}\alpha}{1-\rme^{\pi\frac{\lambda}{\mu}}}\right)A^{-1}\eta\;\;\;\;\mbox{ and }\;\;\;\; P^-_{\alpha, \rho}:=\left(\frac{-\alpha+\rme^{\pi\frac{\lambda}{\mu}}\rho}{1-\rme^{\pi\frac{\lambda}{\mu}}}\right)A^{-1}\eta.$$

The maps 
$$(\alpha, \rho)\mapsto P^+_{\alpha, \rho}\;\;\;\mbox{ and }\;\;\; (\alpha, \rho)\mapsto P^-_{\alpha, \rho},$$
are continuous and it holds that 
$$\alpha\rightarrow-\infty\;\;\;\mbox{ or }\;\;\;\rho\rightarrow+\infty\;\;\;\implies\;\;\;P^+_{\alpha, \rho}\rightarrow+\infty\;\;\mbox{ and }\;\;\;P^-_{\alpha, \rho}\rightarrow-\infty,$$
on the line $\R\cdot(-A^{-1}\eta)$. Therefore, we obtain:

\begin{proposition}
Any LCS on $\R^2$ whose control range $\Omega$ is unbounded is controllable, if the associated matrix $A$ satisfies $\sigma_A<0$.
\end{proposition}

\begin{remark}
To obtain controllability, the previous result only requires that $\Omega$ is unbounded and not necessarily the whole real line (see \cite{Sontag}).
\end{remark}

Also, by using the fact that 
$$(s, \alpha, \rho)\mapsto\varphi(s, P^+_{\alpha, \rho}, \alpha)\;\;\;\mbox{ and }\;\;\;(s, \alpha, \rho)\mapsto\varphi(s, P^-_{\alpha, \rho}, \rho),$$
are continuous maps, one can easily shows that the map 
$$(\alpha, \rho)\in(-\infty, \nu)\times(\nu, +\infty)\mapsto \overline{D_{\alpha, \rho}},$$
is continuous in the Hausdorff measure.

\end{document}